\documentclass[12pt,a4paper,]{amsart}
\usepackage{amssymb, color, hyperref, times}
\usepackage{amssymb,amsfonts,amstext,amsthm, amsmath}
\usepackage[active]{srcltx} 
\usepackage{eepic}
\usepackage[utf8]{inputenc} 

\usepackage{mathtools}
\usepackage{amsmath,amsfonts}
\usepackage{tikz}
\usetikzlibrary{arrows}
\usepackage{mathtools}
\usepackage{lineno}

\def \R {\mathbb{R}}

\def \N {\mathbb{N}}

\def \ds {\displaystyle}

\newtheorem{theorem}{Theorem}[section]
\newtheorem{lemma}{Lemma}[section]

\newtheorem{proposition}{Proposition}[section]
\newtheorem{definition}{Definition}[section]
\newtheorem{remark}{Remark}[section]

\title[Existence of exponential attractor via the $1$-trajectory method] {Existence of exponential attractor to a family of problems dominated by a perturbation of $p(x)$-laplacian with localized large diffusion via the $l$-trajectories method}

\numberwithin{equation}{section} \numberwithin{theorem}{section}

\address[V. L. Carbone]{Universidade Federal de S\~{a}o
Carlos, Departamento de Matem\'atica, 13565-905 S\~{a}o
Carlos SP, Brazil.}
\email{carbone@dm.ufscar.br}

\address[T. R. S. Couto]{Universidade Federal de S\~{a}o
Carlos, Departamento de Matem\'atica, 13565-905 S\~{a}o
Carlos SP, Brazil.}
\email{thays.trsc@gmail.com}

\date{\today}

\thanks{$^\star$Research partially
supported by CAPES, Brazil}

\begin{document}

\maketitle

\centerline{\scshape Vera Lúcia Carbone, Thays Regina Santana Couto$^\star$}
\medskip
{\footnotesize
 \centerline{ Universidade Federal de S\~{a}o
Carlos}
   \centerline{Departamento de Matem\'atica}
   \centerline{ 13565-905 S\~{a}o
Carlos SP, Brazil}
}


\begin{abstract}
This article is devoted to the study of the existence of an 
exponential attractor for a family of problems, in which diffusion $d_{\lambda}$  blows up in localized regions inside the domain,
\begin{equation*}
  \begin{cases}
  \ds
u_t^\lambda-\mathrm{div}(d_\lambda(x)(|\nabla u^\lambda|^{p(x)-2}+\eta ) \nabla u^\lambda)+
|u^\lambda|^{p(x)-2}u^\lambda=B(u^\lambda), & \mbox{ in } \Omega \\
u^\lambda = 0,  & \mbox{ on } \partial\Omega\\
u^\lambda(0)=u^\lambda_0 \in L^2(\Omega),&
  \end{cases}
\end{equation*}
 and their limit problem
via the $l$-trajectory method.
\end{abstract}


\maketitle


\tableofcontents

\section{Introduction}

The existence of exponential attractors is an important feature for a nonlinear systems of differential equations, although not unique, thanks to the exponential rate of attraction exponential attractors are more robust under perturbations than the global attractor.
Several authors has been studied  the existence of an exponential attractor, see \cite{dalibor, Efendi-miran, temam}.  Málek and Pra$\breve{z}$ák in \cite{MalekPrazak} introduced the method of $l$- trajectories, which enables us to discuss the existence of exponential attractor for a nonlinear system of differential equations written as an abstract evolutionary problem,

 \begin{align*}
 \begin{cases}
 u'(t)=F(u(t)),\quad t>0,\,\,\text{in}\,\, X, \\
 u(0)=u_0,
 \end{cases}
 \end{align*}
 where $X$ is an infinite dimensional Banach space, $F:X\to X$ is a nonlinear operator, and $u_0\in X.$

In \cite{MatsuuraOtani} the authors used the method of $l$- trajectories to construct an exponential attractor for a dynamical system associated with equation

 \begin{equation}\label{eqMa_O}
 \begin{cases}
 u_t-\text{div}(a(x,u,\nabla u)) + f(u)=g(x),\,\, \quad(x,t)\in\Omega\times (0,+\infty)\\
 u(\cdot,t)|_{\partial\Omega}=0,\hspace{4,2cm}\,\, t\in(0,+\infty)\\
 u(x,0)=u_0(x), \hspace{3,9cm} x\in\Omega,
 \end{cases}
 \end{equation}
 where $a: \Omega\times\mathbb{R}^1\times\mathbb{R}^n\to \mathbb{R}^n$  satisfies some properties, among them

\begin{equation}\label{Equ_a}
|a(x,u,\nabla u)- a(x,v,\nabla v)|_{\mathbb{R}^n}\leqslant \beta_0|\nabla u-\nabla v|_{\mathbb{R}^n}+\beta_1|u-v|,
\end{equation}
 which a typical example is
 $$u_t-\text{div}(((|\nabla u|^2+\varepsilon)^{\frac{p-2}{2}}+\eta)\nabla u)+ |u|^qu-|u|^ru = g(x),$$
 where $p\in]1,2[,$ $\varepsilon>0,$ $q>r\geqslant 0.$

In the literature, several authors have studied variations of the $p(x)$- Laplacian equations, in \cite{Niu}, under suitable conditions, the problem $$u_t - \text{div}(|\nabla u|^{p(x)-2}\nabla u) + f(x,u)=g$$ admits an infinite-dimensional global attractor. Simsen, in \cite{jacson}, proves  the existence of a global attractor for a $p(x)$- Laplacian inclusion of the form $$u_t - \text{div}(|\nabla u|^{p(x)-2}\nabla u) + \alpha |u|^{p(x)-2}u\in F(u)+h,$$ $\alpha=0,1.$  In  \cite{CCK} the autors assure continuity of the dynamics in a localized large diffusion problem $$u^{\varepsilon}_t - \text{div}(p_{\varepsilon}(x)\nabla u^{\varepsilon}) + \lambda u ^{\varepsilon}=f(u ^{\varepsilon})$$  subjected to nonlinear Neumann boundary conditions.

In \cite{VCK} the authors study   the asymptotic behavior of the solutions of the  family of equations

\begin{equation}\label{prob_lambda}
  \begin{cases}
  \ds
u_t^\lambda-\mathrm{div}(d_\lambda(x)(|\nabla u^\lambda|^{p-2}) \nabla u^\lambda)+
|u^\lambda|^{p-2}u^\lambda=B(u^\lambda), & \mbox{ in } \Omega \\
u^\lambda = 0,  & \mbox{ on } \partial\Omega\\
u^\lambda(0)=u^\lambda_0 \in L^2(\Omega).&
  \end{cases}
\end{equation}

The existence of an exponential attractor to the problem \eqref{prob_lambda} is still an unresolved issue. We will prove the existence of an exponencial attractor  for a perturbation of the problem \eqref{prob_lambda}, in the more general case, where the constant $p$ is replaced by a function $p(x).$

Let $\Omega \subset \R^n$ be an open, bounded, connected and smooth subset, with $n\geqslant 1$. Consider the following family of problems

\begin{equation}\label{prob_plambda}
  \begin{cases}
  \ds
u_t^\lambda-\mathrm{div}(d_\lambda(x)(|\nabla u^\lambda|^{p(x)-2}+\eta ) \nabla u^\lambda)+
|u^\lambda|^{p(x)-2}u^\lambda=B(u^\lambda), & \mbox{ in } \Omega \\
u^\lambda = 0,  & \mbox{ on } \partial\Omega\\
u^\lambda(0)=u^\lambda_0 \in L^2(\Omega),&
  \end{cases}
\end{equation}
for $\lambda \in (0,1]$,  where $p(\cdot)\in C(\Omega)$ satisfies
$$ 2< p^{-}:= \inf \text{ ess }p\leqslant p(x) \leqslant \sup\text{ ess }p:=p^{+},
$$
 $B:L^2(\Omega)\to L^2(\Omega)$ is globally Lipschitz and $\eta > 0$.

Let $\Omega_0$ be an open subset smooth of $\Omega$ with $\overline{\Omega}_0 \subset \Omega$ e $\Omega_0 = \underset{i=1}{\overset{m}\cup} \Omega_{0,i}$ where $m$ is a positive integer and $\Omega_{0,i}$ are smooth subdomains of $\Omega$ satisfying $\overline{\Omega}_{0,i}\cap\overline{\Omega}_{0,j}=\emptyset$, for $i\neq j$. Define $\Omega_1=\Omega\setminus\overline{\Omega}_0$, $\Gamma_{0,i}=\partial\Omega_{0,i}$ and $\Gamma_0=\underset{i=1}{\overset{m}\cup}\Gamma_{0,i}$ as the boundaries of $\Omega_{0,i}$ and $\Omega_0$, respectively. Notice that $\partial \Omega_1=\Gamma\cup\Gamma_0.$

In addition, the diffusion coefficients  $d_{\lambda}: \Omega \subset \R^n \to(0,\infty)$ are bounded and smooth functions in $\Omega$, satisfying
\begin{align*}
0 < m_0 \leqslant d_{\lambda}(x) \leqslant M_\lambda,
\end{align*}
for all $x\in\Omega$ and $0<\lambda\leqslant 1$. We also assume that the diffusion is large in $\Omega_0$ as $\lambda \to 0$, or more precisely,

\begin{equation}\label{convplambda}
d_{\lambda}(x) \xrightarrow[]{\lambda\to0} \left\{%
\begin{array}{ll}
    d_0(x), & \hbox{ uniformly on }\; {\Omega}_1; \\
    \;\;\;\infty, & \hbox{ uniformly on compact subsets of }\;\Omega_0,
\end{array}%
\right.
\end{equation}
where $d_0:{\Omega}_1\to (0,\infty)$ is a smooth function with $0 < m_0 \leqslant d_0(x) \leqslant M_0$, for all $x\in {\Omega}_1$.


If in a reaction-diffusion process the diffusion coefficient behaves as expressed in \eqref{convplambda}  we expect that the solutions of \eqref{prob_plambda} will become  approximately constant on $\Omega_0$. For this reason, suppose that  $u^\lambda$ converges to $u$ as $\lambda \to 0$, in some sense, and that $u$ takes, on $\Omega_0$, a time dependent spatially constant value, which we will denote by $u_{\Omega_0}(t)$.

In this context we will obtain the equation that describes the limit problem.  Notice that, since the limit function $u$ is in $W^{1,p(x)}(\Omega)$, its constant value in $\Omega_0$, $u_{\Omega_0}(t)$ cannot be arbitrary. Also, in the boundary $\Gamma_0=\partial \Omega_0$, we must be $u_{\mid_{\Gamma_0}}=u_{\Omega_0}(t)$. In $\Omega_1$,we have
\begin{equation}\label{EqOmega1}
u_t^\lambda - \mathrm{div}(d_{\lambda}(x)(|\nabla u^\lambda|^{p(x)-2}+\eta)\nabla u^\lambda) + |u^\lambda|^{p(x)-2}u^\lambda
= B(u^\lambda).
\end{equation}
From properties of  convergence of the  function $d_{\lambda}(x)$ in $\Omega_1$, when $\lambda\to0$, we get
$$
u_t - \mathrm{div}(d_0(x)(|\nabla u^\lambda|^{p(x)-2}+\eta)\nabla u^\lambda) + |u|^{p(x)-2}u =
B(u), \hbox{ for}\,\, u \in W^{1,p(x)}(\Omega).
$$
Integrating \eqref{EqOmega1} on $\Omega_0$, from Gauss's  Divergence Theorem it follows that

$$
\underset{\Omega_0}\int u_t^\lambda\,\mathrm{d}x +
\underset{\Gamma_0}\int d_{\lambda}(x) (|\nabla u^\lambda|^{p(x)-2}+\eta)
\frac{\partial u^\lambda}{\partial \vec n } \,\mathrm{d}x +
\underset{\Omega_0}\int |u^\lambda|^{p(x)-2} u^\lambda \,\mathrm{d}x =
\underset{\Omega_0}\int B(u^\lambda)\,\mathrm{d}x,
$$
where $\vec{n}$ denotes the unit inward normal to $\Omega_0$ in the surface integral.
Taking the limit as $\lambda \to 0$, we get the following ordinary differential equation
$$
\dot{u}_{\Omega_0}(t) + \frac{1}{|\Omega_0|}
\left(\underset{\Gamma_0}\int d_0(x)|\nabla u|^{p-2}\frac{\partial
u}{\partial \vec n} \,\mathrm{d}x + \int_{\Omega_0} |u_{\Omega_0}(t)|^{p-2}
u_{\Omega_0}(t) \,\mathrm{d}x\,\right)= B(u_{\Omega_0}(t)).
$$

With these considerations we can write the limiting problem in the following way

\begin{align}\label{prob_pzero}
\begin{cases}
 u_t -\mathrm{div}(d_0(x)(|\nabla u|^{p(x)-2}+\eta)\nabla u) + |u|^{p(x)-2}u=B(u), &    \hbox{ in } \Omega_1 \\ %
 u_{|_{\Omega_{0,i}}}  :=  u_{\Omega_{0,i}},& \hbox{ in}\, \Omega_{0,i}\\
 \ds{\dot{u}_{\Omega_{0,i}}}\!\! + \frac{1}{|\Omega_{0,i}|}
\left[\underset{\Gamma_{0,i}}
    \int\!\! d_0(x)(|\nabla u|^{p(x)-2}+\eta)\frac{\partial u}{\partial \vec n} \,\mathrm{d}x + \underset{\Omega_{0,i}}\int
|u_{\Omega_{0,i}}|^{p(x)-2}u_{\Omega_{0,i}} \,\mathrm{d}x
    \right]= B(u_{\Omega_{0,i}}) &
     \\
 u=0, &  \hbox{ on } \partial\Omega\\
 u(0)=u_0. &
 \end{cases}
\end{align}

 The notion of exponential attractor has been proposed in \cite{Eden-foias-nico-temam}:

\begin{definition} Let $(\rm{E}, d_{\rm{E}})$ be a metric space. A subset $\mathcal{M}\subset\rm{E}$ is an exponential attractor for a semigroup $\{A(t);\, t\geqslant 0\}$ if $\mathcal{M}\neq \emptyset$ is compact, has finite fractal dimension $\rm{dim_f}(\mathcal{M})<\infty$, it is semi-invariant, that is, $A(t)\mathcal{M}\subset\mathcal{M}$ for all $t\geqslant 0,$ and for every bounded subset $D\subset \rm{E}$ there exist constants $c_1, c_2>0$ such that
$$\rm{dist_H}(A(t)D,\mathcal{M})\leqslant c_1e^{-c_2t}, \,\,\text{for all}\,\, t\geqslant 0,$$
where $\rm{dim_f}(A)=\underset{\varepsilon\to 0}\limsup\,\,\displaystyle\frac{\log \rm{N}_{\varepsilon}^{E}(A)}{\log{(\frac{1}{\varepsilon}})},$ and $\rm{N}_{\varepsilon}^{E}(A)$ denotes the minimal number of $\varepsilon$- balls in the space $\rm{E}$ with centers in $A$ needed to cover the subset $A\subset \rm{E}$. Here, $\rm{dist_H}(\cdot, \cdot)$ is the Hausdorff semidistance in $\rm{E}$; that is $\rm{dist_H}(A, B)=\underset{a\in A}\sup\underset{b\in B}\inf d_{\rm{E}}(a,b)$.
\end{definition}

Compared with the global attractor, an exponential attractor is expected to be more robust to perturbations. However that, differently from to the global attractor, an exponential attractor is not necessarily unique, so that its construction relies upon an algorithm. Our contribution in this paper is to ensure, via the $l$-trajectory method, the existence of an exponential attractor for \eqref{prob_plambda}-\eqref{prob_pzero} or, more generally, for evolution equations of the form
\begin{equation*}
  \begin{cases}
  \ds
u_t-\Delta_{\eta,p(x)}u+ |u|^{p(x)-2}u=B(u), & \mbox{ in } \Omega \\
u = 0,  & \mbox{ on } \partial\Omega\\
u(0)=u_0 \in L^2(\Omega),&
  \end{cases}
\end{equation*}
with $\Delta_{\eta,p(x)}u= \mathrm{div}(d(x)(\eta+|\nabla u|^{p(x)-2}) \nabla u)$ where $\eta d(\cdot)$ is the rate of linear diffusion, while $d(\cdot)$ and $p(\cdot)$ allows for capturing the nonlinear diffusion. This problem involves variable exponents which often appears in applications in electrorheological fluids \cite{Ruzicka,Rajagopal} and image processing \cite{Baravdish, Zhichang}.

Although the property \eqref{Equ_a} is not satisfied for $$a(x, u, \nabla u )=d(x)(|\nabla u|^{p(x)-2}+\eta)\nabla u,$$ since $p(x)>2$, it is possible to estimate $|\nabla u|$ and $|u|$ for almost all $ (t,x)\in [0,T]\times\Omega$, where $u$ is solution of  \eqref{prob_pzero} and still use the method of $l$-trajectories to ensure the existence of an exponential attractor for the problem  \eqref{prob_pzero}, similarly to \eqref{prob_plambda}.  

With this in mind, the paper is organized as follows. In Section 2 we define the operators $A_{\lambda}$ and $A_0$,  from the main part of equations, and their properties allows us to ensure the existence of a strong solution to  \eqref{prob_plambda} and \eqref{prob_pzero}. We find uniform estimates for the solutions in Section 3, in special the  Lemma 3.4 which gives us the conditions to proof  Propositon 4.1, a fundamental importance result on Section 4 to guarantee the finite fractal dimension of the attractor.


In section 4 we prove the main result of the paper.
\begin{theorem}\label{Teorema-existencia-exponential-attractor}
 The dynamical systems associate to \eqref{prob_plambda}-\eqref{prob_pzero} possesses the global attractor $\mathcal{A}_\lambda$, for all $\lambda\in[0,1]$,  which is bounded in $L^2(\Omega)$. Moreover, for each $\lambda\in[0,1]$ there exists a positively invariant subset $B$ of $L^2(\Omega)$ such that $\mathcal{A}_{\lambda}\subset B$ and the dynamical systems $(\{T_{\lambda}(t)\}_{t\geqslant 0}, B)$ admits an exponential attractor $\mathcal{E}_\lambda$.
\end{theorem}

\section{Existence of solutions}

In this section we present the operators associated  with the problems \eqref{prob_plambda} and \eqref{prob_pzero}, establish some of its properties and we assures the existence of a unique solution for \eqref{prob_plambda}-\eqref{prob_pzero}.

 We will consider the following spaces and notations.

 \begin{align*}
   &V:= W^{1, p(x)} _{0}(\Omega), \quad\, V_0:=W^{1, p(x)} _{\Omega_{0},0}(\Omega):= \{ u \in W ^ {1,p(x)}_0 (\Omega): u \hbox { is constant in } \Omega_{0} \},  \\
   & H:= L^{2}(\Omega), \qquad\quad H_0:=L^2_{\Omega_{0}}(\Omega):=\{ u \in L^2 (\Omega): u \hbox{ is constant in } \Omega_{0} \}.
   \end{align*}

The space $V_0$ is equipped with the norm 
$$\|v\|_V:= \|v\|_{L^{p(x)}(\Omega)} + \|\nabla v\|_{L^{p(x)}(\Omega)}, $$
where $L^{p(x)}(\Omega)=\{v:\Omega\to \R; v \text{ is measurable and }\int_{\Omega}|v(x)|^{p(x)}\,dx<\infty\}$ is a Banach space with the norm

$$\|v\|_{L^{p(x)}(\Omega)}=\|v\|_{p(x)}:= \inf\left\{ \lambda>0: \rho\biggl(\frac{v}{\lambda}\biggr)\leqslant 1\right\}
$$
with $\rho_{p}(v)$, or simply $\rho(v),$ denoting $ \int_{\Omega}|v(x)|^{p(x)}\,dx$.

 For $v\in L^{p(x)}(\Omega)$ it is important to emphasize the following estimates
\begin{equation}\label{propriedd.p(x).1}
  \min\{\rho(v)^{\frac{1}{p^{-}}}, \rho(v)^{\frac{1}{p^{+}}} \}\leqslant \|v\|_{p(x)}\leqslant\max\{\rho(v)^{\frac{1}{p^{-}}}, \rho(v)^{\frac{1}{p^{+}}}\}.
\end{equation}

Note that $V$ and $V_0$ are reflexive Banach spaces, $V$ is dense in  Hilbert space  $H$ and $V_0$ is dense in Hilbert space $H_0$. Moreover  $V\hookrightarrow\hookrightarrow H\hookrightarrow V'$ which implies $V_0\hookrightarrow\hookrightarrow H_0\hookrightarrow V_0'.$

Consider the operators $A_\lambda: V\to V^{'}$ and $A_0: V_0\to V_0^{'}$ given, respectively, by

\begin{align*}
  \langle A_\lambda u,v\rangle_{V',V}&:=\int_{\Omega}-\text{div} (d_\lambda(x)(|\nabla u|^{p(x)-2}+\eta))v\, dx + \int_{\Omega} |u|^{p(x)-2} u v\, dx\\
  &\,\,=\underset{\Omega}\int d_\lambda(x)(|\nabla u|^{p(x)-2}+\eta)\nabla u\nabla v\,\,dx + \underset{\Omega}\int |u|^{p(x)-2} u v\,\,dx, \quad \forall v\in V,
\end{align*}

\begin{align*}
  \langle A_0 u,v\rangle_{V_0^{'},V_0}:=\underset{\Omega_1}\int -\text{div} (d_0(x)(|\nabla u|^{p(x)-2}&+\eta)\nabla u) v\,\,dx + \underset{\Omega}\int |u|^{p(x)-2} u v \hspace{2cm}\\&+ \int_{\Gamma_0}d_0(x)(|\nabla u|^{p(x)-2}+\eta)\frac{\partial u}{\partial\vec{n}}v\,\,dx, \quad \forall v\in V_0.
\end{align*}

Following the same ideas from \cite{VCK}, we can prove that
\begin{theorem}\label{Teo:H2Artigo2}
For $\lambda \in [0,1]$ the operator $A_\lambda$ is monotone, hemicontinuous and coercive.
\end{theorem}

We define the sets
\begin{align*}
  &D(A_\lambda^H):=\{v \in V : A_\lambda v \in H\}, \quad \hbox{for} \quad \lambda \in (0,1], \\
  &D(A_0^{H_0}):=\{v \in V_0 : A_0 v \in H_0\},
\end{align*}
and consider  the operators $A^H_\lambda:D(A^H_\lambda)\subset H \to H$ given by
\begin{align*}
   & A_\lambda^H(u) = A_\lambda u, \quad\forall u \in D(A_\lambda^H), \quad \hbox{for} \quad \lambda \in (0,1] \hbox{ and} \\
   & A_0^{H_0}(u) = A_0 u,  \quad \forall u \in D(A_0^{H_0}).
\end{align*}

\begin{proposition}\label{Proposicao1.5Dissert}
Assume that $H$ is a Hilbert space. Let $V$ the reflexive Banach space, such that $V\hookrightarrow H \hookrightarrow V'$, moreover $V$ is dense in $H$. If $A: V\to V'$ is monotone, hemicontinuous and coercive then $A_H:D(A_H)\subset H \to H$ is maximal monotone operator.
\end{proposition}

It follows from Proposition \ref{Proposicao1.5Dissert} that $A_\lambda^H $ and $A_0^{H_0}$ are maximal monotones operators.  In addition, these operators can also be seen as subdifferential type, meaning that, $A^H_\lambda = \partial \varphi^\lambda$, where $\varphi^\lambda : H \to (-\infty, \infty]$ are lower semicontinuous convex functions, defined by
\begin{equation*}
 \varphi^\lambda(u) = \left\{
  \begin{array}{ll}
   \ds \int_\Omega \frac{d_\lambda(x)}{p(x)} |\nabla u|^{p(x)}\,\mathrm{d}x +\int_\Omega \frac{d_\lambda(x)\eta}{2} |\nabla u|^{2}\,\mathrm{d}x+  \int_\Omega \frac{1}{p(x)} |u|^{p(x)}\,\mathrm{d}x, & \hbox{ if } u \in V \\
   \infty , & \hbox{ otherwise}
  \end{array}
\right.
\end{equation*}
for $\lambda \in (0,1]$.

For $\lambda=0$, $A_0^{H_0}= \partial \varphi$ where $\varphi : H_0 \to (-\infty, \infty]$ is a lower semicontinuous convex function, defined by
\begin{equation*}
 \varphi(u) = \left\{
  \begin{array}{ll}
   \ds  \int_{\Omega_1} \frac{d_0(x)}{p(x)}|\nabla u|^{p(x)}\,\mathrm{d}x +\int_{\Omega_1} \frac{d_0(x)\eta}{2}|\nabla u|^{2}\,\mathrm{d}x +  \int_\Omega \frac{1}{p(x)} |u|^{p(x)}\,\mathrm{d}x, & \hbox{ if } u \in V_0\\
   \infty , & \hbox{ otherwise}.
  \end{array}
\right.
\end{equation*}
The problems \eqref{prob_plambda} and \eqref{prob_pzero} can be written abstractly as
\begin{equation}\label{prob_P_lamba_abstrato}
\begin{cases}
u_t^{\lambda} + A_{\lambda}u^{\lambda}=B(u^{\lambda}) \\
u^{\lambda}(0)=u_0^{\lambda}, \qquad\qquad \hbox{ for all}\quad \lambda \in (0,1],
\end{cases}
\end{equation}
and
\begin{equation}\label{prob_P_0_abstrato}
\begin{cases}
u_t + A_0u=B(u) \\
u(0)=u_0.
\end{cases}
\end{equation}

\begin{definition}
Let $T>0$, we say that $u \in C([0,T];H)$ is a strong solution to \eqref{prob_P_lamba_abstrato}, if
\begin{enumerate}
   \item [(i)] $u$ is absolutely continuous in any compact subinterval  of $(0,T)$;
 \item  [(ii)] $u(t)\in D(A_\lambda^H)$ almost always in $(0,T)$, $u(0)=u_0$;
 \item [(iii)] $\ds\frac{du}{dt}(t) + A_\lambda^H(u(t)) = B(u(t))$, occurs for almost all  $t\in (0,T)$.
  \end{enumerate}
A function $u \in C([0,T];H)$ is a weak solution to \eqref{prob_P_lamba_abstrato}, if there is a  sequence $(u_n)_{n\in\N}$ of strong solutions convergent to  $u$ in $C([0,T];H)$.
\end{definition}

It follows  from  \cite[Proposition 1]{AJT-ExisteSol}  that \eqref{prob_P_lamba_abstrato} has a global weak solution $u^\lambda(\cdot, u_0^{\lambda})$ starting in $u^\lambda(0)=u_0^{\lambda}\in \overline{D(A_\lambda^{H})}^{H}$. If $u_0^{\lambda} \in D(A_\lambda^H)$ then the functions $u^\lambda(\cdot,u_0^{\lambda})$ are strong solution of \eqref{prob_P_lamba_abstrato} Lipschitz continuous. Analogously for \eqref{prob_P_0_abstrato}.

For $\lambda \in (0,1]$, so we can define in $\overline{D(A_\lambda^H)}^H$ a semigroup $\{T_\lambda(t)\}_{t \geqslant 0}$ of nonlinear operators, associated with \eqref{prob_P_lamba_abstrato} by $T_\lambda(t)u_0^{\lambda} = u^\lambda(t,u_0^{\lambda})$, $t\geqslant 0$.

 To simplify, we will denote the solution $u^0(t,u_0)$ of \eqref{prob_P_0_abstrato}  just by $u(t,u_0)$. Thus, if $\lambda=0$, we can define in $\overline{D(A_0^{H_0})}^{H_0}$ a semigroup $\{T(t)\}_{t \geqslant 0}$ of nonlinear operators, associated with \eqref{prob_P_0_abstrato} by $T(t)u_0 = u(t,u_0)$, $t\geqslant 0$.

\section{Estimates involving the solution}
The estimates of the solution represents a important step for assure the existence of an absorbing ball in $H$ for the dynamical system $(\{T_\lambda(t)\}_{t\geqslant 0}, H)$ this is one of the purpose of this section.

\begin{lemma}\label {estimativapot}
Let $\lambda, \mu$ be arbitrary nonnegative numbers. Then for all $\alpha,\beta, \alpha\geqslant \beta \geqslant 0$
\begin{equation}
\lambda ^{\alpha}+ \mu ^{\beta}\geqslant \frac{1}{2^{\alpha}}
\begin{cases}
(\lambda + \mu)^{\alpha},\,\, if \lambda + \mu <1,\\
(\lambda + \mu)^{\beta},\,\, if \lambda + \mu \geqslant 1,
\end{cases}
\end{equation}

\begin{equation}
\lambda ^{\alpha}+ \mu ^{\beta}\leqslant 2
\begin{cases}
(\lambda + \mu)^{\alpha},\,\, if \lambda + \mu \geqslant 1,\\
(\lambda + \mu)^{\beta},\,\, if \lambda + \mu <1,
\end{cases}
\end{equation}
\end{lemma}

\begin{proof}
See the Proposition 3.1 in \cite{jacson_claudianor}.
\end{proof}

\begin{lemma}\label{lema.limitacao_u_em_V}

   Let $u$ be a solution of \eqref{prob_P_0_abstrato} . Then

   \begin{enumerate}
   \item There is a positive constant $r_0$  such that
  $\|u(t)\|_H\leqslant r_0, \,\,\, \text{for all}\,\, t \geqslant 1.$
  \item There is a positive constant $r$  such that
  $\|u(t)\|_V\leqslant r, \,\,\, \text{for all}\,\, t \geqslant 2.$
  \end{enumerate}
We obtain the same estimates if $u^{\lambda}$  is a solution of \eqref{prob_P_lamba_abstrato}, uniformly in $(0,1]$.

\end{lemma}
\begin{proof} Let $u$ be a solution of \eqref{prob_P_0_abstrato} with $\|u\|_{V_0}\geqslant 1$. Taking the scalar product with $u(t)$ in \eqref{prob_P_0_abstrato}, we get

\begin{align*}
\frac{1}{2}\frac{d}{dt}\|u(t)\|^2_{H} + \langle A_0u(t), u(t)\rangle_{V_0',V_0}&=\biggl(\frac{du}{dt}(t), u(t)\biggr)_{L^2} + (A_0u(t), u(t))_{L^2}\\&=(B(u(t)),u(t))_{L^2}.
\end{align*}
So,
\begin{align*}
\frac{1}{2}\frac{d}{dt}\|u(t)\|^2_{H}+\frac{C}{2^{p^+}}\|u(t)\|_{V_0}^{p^-}&\leqslant \frac{1}{2}\frac{d}{dt}\|u(t)\|^2_{H} + \langle A_0u(t), u(t)\rangle_{V_0',V_0}\\
&\leqslant L \|u(t)\|_{H}^2+\|B(0)\|_{H}\|u(t)\|_{H},
\end{align*}
as $V_0\hookrightarrow\hookrightarrow H$ we have $\|u(t)\|_{H}\leqslant \eta\|u(t)\|_{V_0}$, where $\eta=|\Omega|+1$, so that
\begin{align*}
\frac{1}{2}\frac{d}{dt}\|u(t)\|^2_{H}+\frac{C}{2^{p^+}}\|u(t)\|_{V_0}^{p^-}&\leqslant c_1\|u(t)\|^2_{V_0} + c_2 \|u(t)\|_{V_0}
\end{align*}
where $c_1=L\eta^2$ and $c_2=\eta\|B(0)\|_{H}.$

Consider  $\theta=\frac{1}{2}p^-$ and $\varepsilon>0$, chosen so that $\frac{c}{2^{p^+}}-\frac{1}{\theta}\varepsilon^{\theta}-\frac{1}{p^-}\varepsilon^{p^-}>0$. It follows from Young's inequality that
\begin{align*}
c_1\|u(t)\|^2_{V_0} + c_2 \|u(t)\|_{V_0} \leqslant \frac{1}{\theta}\varepsilon^{\theta}\|u(t)\|_{V_0} ^{p^-} + \frac{1}{\theta'}\biggl(\frac{c_1}{\varepsilon}\biggr)^{\theta'} + \frac{1}{p^-}\varepsilon^{p^-}\|u(t)\|_{V_0}^{p^-}+\frac{1}{q}\biggl(\frac{c_2}{\varepsilon}\biggr)^{q},
\end{align*}
 where $q=(p^-)'.$ Let $\gamma=\frac{c}{2^{p^+}}-\frac{1}{\theta}\varepsilon^{\theta}-\frac{1}{p^-}\varepsilon^{p^-} >0$ then
\begin{align*}
\frac{d}{dt}\|u(t)\|^2_{H}+\frac{2\gamma}{\eta^{p^-}}\|u(t)\|^{p^-}_{H}&\leqslant\frac{d}{dt}\|u(t)\|^2_{H}+2\gamma\|u(t)\|_{V_0}^{p^-}\\
&\leqslant \frac{2}{\theta'}\biggl(\frac{c_1}{\varepsilon}\biggr)^{\theta'} +\frac{2}{q}\biggl(\frac{c_2}{\varepsilon}\biggr)^{q}, \forall t\geqslant 0,
\end{align*}
taking $\delta=\frac{2}{\theta'}\biggl(\frac{c_1}{\varepsilon}\biggr)^{\theta'} +\frac{2}{q}\biggl(\frac{c_2}{\varepsilon}\biggr)^{q},$ $\tilde{\gamma}=\frac{2\gamma}{\eta^{p^-}}$ and $y(t)=\|u(t)\|^2_{H}$ we have
$$\frac{d}{dt}y(t)+\tilde{\gamma}y(t)^{\frac{p^-}{2}}\leqslant \delta,\,\,\,\forall t\geqslant 0.$$
Now, applying  \cite[Lemma 5.1]{temam},  we get 
\begin{align*}
\|u(t)\|_{H}\leqslant \biggl(\frac{\delta}{\tilde{\gamma}}\biggr)^{\frac{1}{p^-}} + \biggl(\tilde{\gamma}\biggl(\frac{p^--2}{2}\biggr)t\biggr)^{-\frac{1}{p^--2}}, \,\,\,\forall t\geqslant 0.
\end{align*}
Therefore,  $$\|u(t)\|_{H}\leqslant k_1, \,\,\,t\geqslant 1$$ where $k_1=\biggl(\frac{\delta}{\tilde{\gamma}}\biggr)^{\frac{1}{p^-}} + \biggl(\tilde{\gamma}\biggl(\frac{p^--2}{2}\biggr)\biggr)^{-\frac{1}{p^--2}}.$ For $t,$ such that $\|u(t)\|_{V_0}\leqslant 1$, we have $$\|u(t)\|_{H}\leqslant\eta\|u(t)\|_{V_0}\leqslant \eta.$$ Consequently,
$$\|u(t)\|_{H}\leqslant r_0, \,\,\,\forall t\geqslant 1,$$
where $r_0=\max\{k_1,\eta\}.$

On the other hand, 

 \begin{align*}
 \frac{d}{dt}\varphi(u)=\langle\partial\varphi(u), u_t\rangle_H=\langle A_0 u, u_t\rangle=\langle B(u)-u_t, u_t\rangle=-\|B(u)-u_t\|_H^2+\langle B(u)-u_t, B(u)\rangle.
 \end{align*}

 With this, we ensure that

 \begin{align*}\frac{d}{dt}\varphi(u)\leqslant \frac{1}{2}\|B(u)\|_H^2\leqslant \frac{1}{2}k_2^2, \,\,\,\forall t\geqslant 1,
 \end{align*}
where $k_2:=Lr_0+\|B(0)\|_H.$

 It follows from the definition of subdifferential that

 \begin{align}\label{estimativa varphi}
 \frac{1}{2}\frac{d}{dt}\|u\|_H^2 + \varphi (u)&=\langle u_t, u \rangle _ H + \varphi (u)\leqslant \langle u_t, u \rangle _ H  + \langle \partial\varphi(u), u\rangle _H\\ \nonumber
 & \leqslant \|B(u)\|_H \|u\|_H\leqslant k_2r_0, \,\, \forall t\geqslant 1.
 \end{align}

 Integrating the estimate \eqref{estimativa varphi} in $[t, t+1],$ $t\geqslant 1$, we get

 \begin{align*}
 \int_t^{t+1} \varphi (u)\, ds&\leqslant \frac{1}{2}\|u(t+1)\|_H^2 + \int_t^{t+1} \varphi (u)\, ds\\
 &\leqslant \frac{1}{2} \|u(t)\|_H^2 +k_2r_0\leqslant\frac{1}{2} r_0^2 + k_2r_0:=k_3.
 \end{align*}
 Using the Uniform  Gronwall Lemma, \cite[Lemma 1.1]{temam},  for $y=\varphi (u)$, $g=0$ and $h=\frac{1}{2} k_2^2$ we conclude that
 \begin{align*}
 \varphi (u(t+1))\leqslant k_3+\frac{1}{2}k_2^2:=k_4,\,\,\forall t\geqslant 1.
 \end{align*}
 Thus, we have

 \begin{align}\label{estimativarho1}
 \frac{1}{p^+}\min\{m_0,1\} (\rho(u)+ \rho (\nabla u))\leqslant 2^{p^+} \varphi (u(t))\leqslant 2^{p^+}k_4, \,\,\forall t\geqslant 2.
 \end{align}

If $\|u\|_{V}\leqslant 1$, there is nothing to prove. If $\|u\|_V\geqslant 1$ we have two cases :
  \begin{enumerate}
  \item $\|u\|_{p(x)}\leqslant 1$ and  $\|\nabla u\|_{p(x)}\geqslant 1$, using \eqref{propriedd.p(x).1} we obtain
 $$ \|u\|_{p(x)}^{p^{+}}\leqslant \rho (u)\leqslant \|u\|_{p(x)}^{p^{-}}\quad \text{and}\quad \|\nabla u\|_{p(x)}^{p^{-}}\leqslant \rho (\nabla u)\leqslant \|\nabla u\|_{p(x)}^{p^{+}}.$$
Follows from Lemma \ref{estimativapot} that

\begin{align*}
\|u\|_{p(x)}^{p^{+}} + \|\nabla u\|_{p(x)}^{p^{-}}\geqslant \frac{1}{2^{p^+}} (\|u\|_{p(x)}+ \|\nabla u\|_{p(x)})^{p^{-}}
\end{align*}
 consequently,
\begin{align}\label{estimativarho}
\frac{1}{2^{p^+}}\|u\|_{V}^{p^{-}}\leqslant \rho (u) + \rho (\nabla u).
\end{align}

 Thanks to \eqref{estimativarho1} and \eqref{estimativarho} we conclude that
$$\|u\|_{V}\leqslant 2^{\frac{p^{+}}{p^{-}}}
\biggl(\frac{p^{+}k_4}{\min\{m_0, 1\}}\biggr)^{\frac{1}{p^{-}}}.$$

   \item $\|u\|_{p(x)}\geqslant 1$ and  $\|\nabla u\|_{p(x)}\leqslant 1$. Follows similarly to item $(1)$.
\end{enumerate}
 Therefore $$\|u\|_{V}\leqslant  \max\biggl\{
\biggl(\frac{2^{p^{+}}p^{+}k_4}{\min\{m_0, 1\}}\biggr)^{\frac{1}{p^{+}}}, \biggl(\frac{2^{p^{+}}p^{+}k_4}{\min\{m_0, 1\}}\biggr)^{\frac{1}{p^{-}}}, 1\biggr \},$$ forall $ t\geqslant 2.$

\end{proof}

The nexts results assures about  the estimate of $|\nabla u_\lambda(t,x)|$ and $|u_\lambda(t,x)|$, for almost all $(t,x)\in [0,T]\times\Omega$ and for all $\lambda \in [0,1]$, this will help  prove a strong continuity of a shift operator defined in the next section.

\begin{lemma}\label{lema.Estimativa_Uniforme.Ellard}
  Let $p(\cdot)$ a function in $L^{\infty}(\Omega)$, with $1< p^{-}\leqslant p(x)$ and $f:[0,T]\times \Omega \to \R$, $f(t,\cdot) \in L^{p(x)}(\Omega)$. If exists $C>0$ an independent constant of $p(\cdot)$, such that
  $$ \int_{0}^{T}\int_{\Omega}|f(t,x)|^{p(x)}\,dx \,dt\leqslant C,
  $$
  then $|f(t,x)|\leqslant 1$,
almost all $(t,x)\in[0,T]\times\Omega$.
\end{lemma}
\begin{proof}

Fixed $\beta>0,$ we consider the set

  $$ A_\beta=\{(t,x)\in [0,T]\times \Omega;|f(t,x)|\geqslant1+\beta\} \subset [0,T]\times \Omega. $$
Using Hölder's inequality we have

  \begin{align*}
    |A_\beta|(1+\beta) =& \iint_{A_\beta}(1+\beta)\,dx\,dt= \liminf_{p^{-}\to \infty}\iint_{A_\beta}(1+\beta)\,dx\,dt \\
    \leqslant & \liminf_{p^{-}\to \infty}\iint_{A_\beta}|f(t,x)|\,dx\,dt \leqslant \liminf_{p^{-}\to \infty}\left( \frac{1}{p^{-}}+ \frac{1}{q^{-}}\right)\|1\|_{L^{q(x)}(A_\beta)}\|f\|_{L^{p(x)}(A_\beta)},
  \end{align*}
  where $q(x)= \frac{p(x)}{p(x)-1}$. By property \eqref{propriedd.p(x).1} follows that
  \begin{equation*}
    |A_\beta|(1+\beta) \leqslant \liminf_{p^{-}\to \infty}\left( \frac{1}{p^{-}}+1- \frac{1}{p^{+}}\right)\max \{\rho_{q}(1)^{\frac{1}{q^{-}}}, \rho_{q}(1)^{\frac{1}{q^{+}}}\} \max\{\rho_{p}(f)^{\frac{1}{p^{-}}}, \rho_{p}(f)^{\frac{1}{p^{+}}}\}.
  \end{equation*}
  In its turn $\rho_{q}(1)= \iint_{A_\beta} 1^{q(x)} \,dx\,dt=|A_\beta|$ and
  \begin{equation*}
     \rho_{p}(f)= \iint_{A_\beta} |f(t,x)|^{p(x)} \,dx\,dt\leqslant \int_{0}^{T}\int_{\Omega} |f(t,x)|^{p(x)} \,dx\,dt\leqslant C,
  \end{equation*}
  which implies that
  \begin{align*}
    |A_\beta|(1+\beta) \leqslant & \liminf_{p^{-}\to \infty}\left( \frac{1}{p^{-}}+1- \frac{1}{p^{+}}\right)\max \{|A_\beta|^{\frac{1}{q^{-}}}, |A_\beta|^{\frac{1}{q^{+}}}\} \max\{C^{\frac{1}{p^{-}}}, C^{\frac{1}{p^{+}}}\} \\
    = & \max \{\liminf_{p^{-}\to \infty}|A_\beta|^{1-\frac{1}{p^{+}}}, \liminf_{p^{-}\to \infty} |A_\beta|^{1-\frac{1}{p^{-}}}\} \max\{\liminf_{p^{-}\to \infty}C^{\frac{1}{p^{-}}}, \liminf_{p^{-}\to \infty}C^{\frac{1}{p^{+}}}\} = |A_\beta|,
  \end{align*}
  so, $|A_\beta|(1+\beta) \leqslant |A_\beta|$ and with that $|A_\beta|=0$, for all $\beta >0$, implying that $|f(t,x)|\leqslant 1,$ almost all  $(t,x)\in [0,T]\times \Omega$.
\end{proof}

\begin{lemma}\label{lema1-2} 
  Let $u$ solution of \eqref{prob_P_0_abstrato}  and $T>0$, then $|\nabla u(t,x)|\leqslant 1$ and $|u(t,x)|\leqslant 1,$ almost all $(t,x)\in [0,T]\times\Omega$.
\end{lemma}
\begin{proof}
  Since $u$ is a solution of \eqref{prob_P_0_abstrato}, then $ u_t + A_0^{H_0}u=Bu$.
Making $(\cdot ,u)_{H_0}$ follow that
  \begin{equation}\label{lema1-2.eq1}
    (u_t,u)_{H_0} + (A_0^{H_0}u,u)_{H_0}=(Bu,u)_{H_0}.
  \end{equation}
  In its turn
  \begin{align*}
    (A_0^{H_0}u,u)_{H_0} =& \langle A_0^{H_0}u,u\rangle_{V_0^{'},V_0}= -\int_{\Gamma_0\cup \Gamma}d_0(x)(|\nabla u|^{p(x)-2}+\eta)\frac{\partial u}{\partial\vec{n}}u\,\,dx+ \underset{\Omega}\int |u|^{p(x)-2} |u|^2\,\,dx\\&+\underset{\Omega_1}\int d_0(x)(|\nabla u|^{p(x)-2}+\eta)\nabla u\nabla u\,\,dx \,\, dx+ \int_{\Gamma_0}d_0(x)(|\nabla u|^{p(x)-2}+\eta)\frac{\partial u}{\partial\vec{n}}u\,\,dx \\
    = & \int_{\Omega_1} d_0(x)|\nabla u|^{p(x)}\,dx+ \int_{\Omega_1} d_0(x)\eta|\nabla u|^2\,dx  +\int_{\Omega}|u|^{p(x)}\,dx,
  \end{align*}
  coming back to \eqref{lema1-2.eq1} we get
  \begin{align*}
    \frac{1}{2}\frac{d}{dt}\|u\|_{H_0}^2 + \int_{\Omega_1} d_0(x)|\nabla u|^{p(x)}\,dx+ \int_{\Omega_1} d_0(x)&\eta|\nabla u|^2\,dx  +\int_{\Omega}|u|^{p(x)}\,dx = (Bu,u)_{H_0}\\
    &\leqslant |(Bu,u)_{H_0}|\leqslant   \|Bu\|_{H_0}\|u\|_{H_0}\leqslant L_B\|u\|_{H_0}^2.
  \end{align*}
  Since
   $$ \int_{\Omega_1}  d_0(x)|\nabla u|^{p(x)}\,dx \geqslant m_0\int_{\Omega_1}|\nabla u|^{p(x)}\,dx\geqslant 0,$$
   $$  \int_{\Omega_1} d_0(x)\eta|\nabla u|^2\,dx\geqslant m_0\eta\|\nabla u\|_{L^2(\Omega_1)}^2\geqslant 0 $$
     and $$\rho_{p}(u)= \int_{\Omega}|u|^{p(x)}\,dx\geqslant0,$$
 we obtain
 
  \begin{equation}\label{lema1-2.eq2.u}
    \frac{1}{2}\frac{d}{dt}\|u\|_{H_0}^2 +\int_{\Omega}|u|^{p(x)}\,dx\leqslant L_B\|u\|_{H_0}^2,
  \end{equation}
  and,
  \begin{equation}\label{lema1-2.eq3.grad u}
    \frac{1}{2}\frac{d}{dt}\|u\|_{H_0}^2 +  m_0\int_{\Omega_1}|\nabla u|^{p(x)}\,dx \leqslant L_B\|u\|_{H_0}^2.
  \end{equation}
  Now integrating \eqref{lema1-2.eq2.u} with respect to $t\in[0,T],$ we have
  \begin{align*}
  &\frac{1}{2} \int_{0}^{T} \frac{d}{dt}\|u(t)\|_{H_0}^2\,dt + \int_{0}^{T}\int_{\Omega}|u(t)|^{p(x)}\,dx\, dt\leqslant \int_{0}^{T}L_B\|u(t)\|_{H_0}^2 \,dt 
  \end{align*}
  thus
  \begin{align*}
  \|u(T)\|_{H_0}^2 + 2\int_{0}^{T}\int_{\Omega}|u(t)|^{p(x)}\,dx\, dt\leqslant 2L_B\int_{0}^{T}\|u(t)\|_{H_0}^2 \,dt + \|u(0)\|_{H_0}^2,
\end{align*}
which implies
\begin{equation}\label{lema1-2.eq4}
  2\int_{0}^{T}\int_{\Omega}|u(t)|^{p(x)}\,dx\, dt\leqslant 2L_B\int_{0}^{T}\|u(t)\|_{H_0}^2 \,dt + \|u(0)\|_{H_0}^2.
\end{equation}
Neglecting  the second term on the left-hand side of \eqref{lema1-2.eq2.u} and integrating over $[s,t],$ where $0\leqslant s<t,$ we obtain

\begin{align}\label{estimativa_u}
   \|u(t)\|_{H_0}^2- \|u(s)\|_{H_0}^2 \leqslant 2L_B\int_{s}^{t}\|u(\theta)\|_{H_0}\,d\theta.
\end{align}
 Using the Gronwall's inequality, \cite[Corollary 6.6]{hale}, in \eqref{estimativa_u} we have
$$ \|u(t)\|_{H_0}^2\leqslant\|u(s)\|_{H_0}^2\exp(2L_B(t-s)), \text{ para }0\leqslant s<t.$$
In particular for $s=0,$ $$\|u(t)\|_{H_0}^2\leqslant\|u(0)\|_{H_0}^2\exp(2L_Bt)$$ and integrating with respect to $t\in[0,T]$, we get
\begin{align*}
  \int_{0}^{T}\|u(t)\|_{H_0}^2\,dt\leqslant & \int_{0}^{T}\|u(0)\|_{H_0}^2\exp(2L_Bt)\,dt \\
  = & \frac{1}{2L_B} \|u(0)\|_{H_0}^2(\exp(2L_BT)-1).
\end{align*}
Returning the equation \eqref{lema1-2.eq4} we can write
\begin{align*}
  \int_{0}^{T}\int_{\Omega}|u(t)|^{p(x)}\,dx\, dt \leqslant \frac{1}{2}\exp(2L_BT)\|u(0)\|_{H_0}^2.
\end{align*}
Similarly, it follows from \eqref{lema1-2.eq3.grad u}, that
\begin{equation*}
  \int_{0}^{T}\int_{\Omega}|\nabla u(t)|^{p(x)}\,dx\, dt \leqslant \frac{1}{2m_0}\exp(2L_BT)\|u(0)\|_{H_0}^2.
\end{equation*}
The result follows from do Lemma \ref{lema.Estimativa_Uniforme.Ellard}.
\end{proof}

\section{Existence of exponential attractor via l-trajectory method} 

The aim of this section is to prove that $(\{T_0(t)\}_{t\geqslant 0}, H_0)$ has an exponential attractor. In particular, this also implies that $(\{T_0(t)\}_{t\geqslant 0}, H_0)$ has a global attractor with finite fractal dimension. 

Consider $B_1= \{u\in D(A_0^{H_0}); \|u\|_{p(x)}\leqslant r \text{ and } \|\nabla u\|_{p(x)}\leqslant r\}$, by Lemma \ref{lema.limitacao_u_em_V} we have that $T_0(t)D(A_0^{H_0})\subset B_1$ for all $t\geqslant 2$. We define
$$B_0= \bigcup_{t\in [0,2]}T_0(t)B_1,
$$
we have to $B_0$ is positively invariant with respect to $\{T_0(t)\}_{t\geqslant 0},$ and is compact in $H$, since $B_1$ is compact.

\begin{remark}
For $\lambda\in ]0,1]$, $B_1= \{u\in D(A_\lambda^H); \|u\|_{p(x)}\leqslant r \text{ and } \|\nabla u\|_{p(x)}\leqslant r\}$ and $B_0= \bigcup_{t\in [0,2]}T_\lambda(t)B_1$.
\end{remark}

Let us denote by $\mathfrak{X}$
the set of all the solutions of \eqref{prob_pzero} defined in the interval $[0,1]$ equipped with the topology of space $L^2(0,1;H)$.


Consider $\{L(t)\}_{t\geqslant 0}$ the semigroup shift of $1$-trajectories,
that is, given $\chi \in \mathfrak{X}$, we defined
$$(L(t)\chi)(\tau)= u(t+\tau), \quad \forall\tau \in [0,1],
$$
where $u$ is the only solution with $\chi= u|_{[0,1]}$.

We will show that the semigroup $\{T_0(t)\}_{t\geqslant 0}$ associated with \eqref{prob_P_0_abstrato} admits exponential attractor via the $l$-trajectories method,
for this, we will consider
$$Y=\left\{\chi \in L^2(0,1;L^2(\Omega)); \chi \in L^2(0,1;V_0) \text{ and } \frac{d\chi}{dt}\in L^2(0,1; V^{'}_0)\right\},
$$
follows from  Aubin-Lions Lemma, see \cite{aubin-lions}, that $Y\hookrightarrow\hookrightarrow L^2(0,1;H_0)$. The $Y$ space is equipped with the norm
$$ \|u\|_{Y}=\|\nabla u\|_{L^2(0,1;L^2(\Omega))} + \|u_{t}\|_{L^2(0,1;V^{'}_0)}.
$$
Define $\mathcal{B}_0 =\{\chi\in \mathfrak{X}; \chi(0)\in B_0\}$,
by definition $\mathcal{B}_0 $ is positively invariant with respect to $\{L(t)\}_{t\geqslant 0}$.

\begin{lemma}
The set $\mathcal{B}_0=\{\chi\in \mathfrak{X}; \chi(0)\in B_0\} $ is compact in $L^2(0,1;H_0)$.
\end{lemma}

\begin{proof}
Note that $\mathcal{B}_0$ is  limited in $\{u\in L^2(0,1;V); \,u_t \in L^2(0,1;H_0)\}$. Indeed, if $\chi\in \mathcal{B}_0$ then $\chi (t)=T(t)\chi(0)$ where $\chi(0)\in D(A)$, using the Theorem 3.17 from \cite{BrezisF}, we have

\begin{align*}
\|\chi_t\|_{L^2(0,1;H_0)}^2=\int_0^1 \|(T(t)\chi(0))_t\|_{H_0}^2\, dt=\|(T(\cdot)\chi(0))_t\|_{L^2(0,1;H_0)}\leqslant k_5.
\end{align*}

In addition, it follows from Poincaré inequality that

\begin{align*}
\|\chi\|_{L^2(0,1;V)}^2=\int_0^1\, (\|\chi(t)\|_{p(x)}+\|\nabla\chi(t)\|_{p(x)})^2\, dt\leqslant C \int_0^1\|\nabla T(t)\chi(0)\|_{p(x)}^2\, dt.
\end{align*}
As $\chi(0)=u_0$ and $2<p^{-}\leqslant p(x)$, if $\|\nabla T(t)u_0\|_{p(x)}\geqslant 1$ we have

\begin{align*}
\|\nabla T(t)u_0\|_{p(x)}^2\leqslant \|\nabla T(t)u_0\|_{p(x)}^{p^{-}}\leqslant \rho_{p(\cdot)}(\nabla T(t)u_0).
\end{align*}
Thus,

\begin{align*}
\|\chi\|_{L^2(0,1;V)}^2\leqslant \int_0^1 \int_{\Omega}\, C |\nabla T(t)u_0|^{p(x)}\, dx dt\leqslant C \frac{e^{2L}}{2m_0}\|u_0\|_{H_0}^2,
\end{align*}

so that $\|\chi\|_{L^2(0,1;V)}^2\leqslant \min \biggl\{C, C\frac{e^{2L}}{2m_0}\|\chi(0)\|_{H_0}^2\biggr\}$.

As $\{u\in L^2(0,1;V); u_t\in L^2(0,1;H)\}\hookrightarrow\hookrightarrow L^2(0,1; H_0)$ we get  $\overline{\mathcal{B}_0}^{H_{0}}$ is compact in $L^2(0,1;H_0)$.

Consider $\{\chi_n\}_{n\in\mathbb{N}}\subset\mathcal{B}_0$ such that $\chi_n\to\chi$ in $L^2(0,1;H_0)$, restricting to a subsequence if necessary we have to $\chi_n(t)\to\chi(t)$, for almost all $t\in [0,1]$. Additionally $\chi_n(t)=T(t)\chi_n(0)\in B_0,$ $\forall n\in\mathbb{N}$, $\forall t\in[0,1]$ and $\chi(t)\in B_0$ for almost all  $t\in[0,1]$, in particular there is a sequence $\{t_n\}_{n\in\mathbb{N}}$ such that $t_n\to 0$ and $\chi(t_n)\in B_0,\, \forall n\in\mathbb{N}$. As $\chi: [0,1]\to L^2(\Omega)$ is continuous then $\chi(0)\in\overline{B_0}=B_0$ from which we conclude that $\mathcal{B}_0$ it is closed in $L^2(0,1;H_0)$.
\end{proof}

\begin{remark}
An adaptation of the demonstration in \cite{Tartar}, page 80, we get the Tartar inequality. Let $x,y\in\mathbb{R}^n$, then

\begin{align}\label{desiTartar}
\langle |x|^{p-2}x-|y|^{p-2}y, x-y\rangle\geqslant
\begin{cases}
\frac{2^{3-p}}{p} |x-y|^p, \,\, \quad\qquad \text{if}\,\,p\geqslant 2\\
(p-1)\frac{|x-y|^2}{(|x|^p+ |y|^p)^{2-p}},\,\,\,\text{if}\,\, 1< p<2.
\end{cases}
\end{align}

\end{remark}

We will now demonstrate the property of Lipschitz for $L(1): L^2(0,1;H_0) \to Y$ in $\mathcal{B}_0$, using  the fact that $z\mapsto |z|^{p(x)-2}$ is locally Lipschitz, for $p(x)>2.$ 
\begin{proposition}
  There is a constant $\rho >0$ such that
  $$ \|L(1)\chi_1- L(1)\chi_2\|_Y \leqslant \rho \|\chi_1 -\chi_2\|_{L^2(0,1;H_0)},
  $$
  for all $\chi_1,\chi_2\in \mathcal{B}_0$.
\end{proposition}
\begin{proof}
  Let $u$ and $v$ solutions of \eqref{prob_P_0_abstrato}  such that $u(0)=u_0\in B_0$ and $v(0)=v_0\in B_0$, then
  \begin{equation*}
    u_t +A_0^{H_0}u =Bu \quad \text{ and } \quad v_t +A_0^{H_0}v =Bv,
  \end{equation*}
  making the difference of the equations and denoting $w=u-v$ we get
  \begin{equation}\label{lema3.eq1}
    w_t +A_0^{H_0}u - A_0^{H_0}v =Bu -Bv.
  \end{equation}
  For $\theta \in [0,1]$ and $\psi \in V_0$
  \begin{align*}
    |(w_t(1+\theta),\psi)_{H_0}| & \leqslant |(A_0^{H_0}u(1+\theta) - A_0^{H_0}v(1+\theta), \psi)_{H_0}|+|(Bu(1+\theta) -Bv(1+\theta),\psi)_{H_0}| \\
     & \leqslant |\langle A_0 u(1+\theta) - A_0 v(1+\theta), \psi\rangle_{V_0^{'},V_0}|+\|Bu(1+\theta) -Bv(1+\theta)\|_{H_0}\|\psi\|_{H_0} \\
     & \leqslant |\langle A_0 u(1+\theta) - A_0 v(1+\theta), \psi\rangle_{V_0^{'},V_0}|+ L_B\|u(1+\theta) -v(1+\theta)\|_{H_0}\|\psi\|_{H_0}.
  \end{align*}
 where
\begin{align*}
  |\langle A_0  u(1+\theta) &- A_0 v(1+\theta), \psi\rangle_{V_0^{'},V_0}| \leqslant  \underset{\Omega}\int | |u(1+\theta)|^{p(x)-2}  u(1 +\theta)\\
  & \hspace{7,3cm} - |v(1+\theta)|^{p(x)-2}v(1+\theta)|\,|\psi | \,dx\\
  &+ \underset{\Omega_1}\int |d_0(x)[(|\nabla u(1+\theta)|^{p(x)-2}+\eta)\nabla u(1+\theta)\\
  & \hspace{5cm} -(|\nabla v(1+\theta|)^{p(x)-2}+\eta)\nabla v(1+\theta)]||\nabla \psi| \,dx.
\end{align*}
Using the Lemma \ref{lema1-2} and the fact $z\to |z|^{p(x)-2}z$ be locally lipschitz, since $p(x)>2$, there exists $\beta>0$ so that
\begin{align*}
  |\,\,|u(1+\theta)|^{p(x)-2}u(1+\theta)- |v(1+\theta)|^{p(x)-2}v(1+\theta)|\leqslant & \beta |u(1+\theta)- v(1+\theta)|\quad \text{ and} \\
 | \,\,|\nabla u(1+\theta)|^{p(x)-2}\nabla u(1+\theta)- |\nabla v(1+\theta)|^{p(x)-2}\nabla v(1+\theta)| \leqslant & \beta|\nabla u(1+\theta)- \nabla v(1+\theta)|,
\end{align*}
almost all $(t, x)\in [0,1]\times\Omega$.
By  Hölder's inequality,  in $H_0$, we have
\begin{align*}
  |\langle A_0 u(1+\theta) - A_0 v(1+\theta), \psi\rangle_{V_0^{'},V_0}| \leqslant M_0(\beta+\eta)\|\nabla (u(1 &+\theta) -v(1+\theta))\| \|\nabla\psi\|\\
  & + \beta\|u(1+\theta) -v(1+\theta)\|\|\psi\|,
\end{align*}
and so in space $H_0$ we obtain
\begin{align*}
  |(w_t(1+\theta),\psi)_{H_0}|\leqslant M_0(\beta+\eta)\|\nabla (u(1+\theta)  &-v(1+\theta))\| \|\nabla\psi\|\\& + (\beta +L_B)\|u(1+\theta) -v(1+\theta)\|\|\psi\|.
\end{align*}
It follows from Poincaré's inequality there exists $\alpha>0$ such that
\begin{equation*}
  |(w_t(1+\theta),\psi)_{H_0}| \leqslant (M_0(\beta +\eta) + \alpha^2 (\beta +L_B)) \|\nabla (u(1+\theta) -v(1+\theta))\|_{H_0} \|\nabla\psi\|_{H_0},
\end{equation*}
with $\gamma = M_0(\beta +\eta) + \alpha^2 (\beta +L_B)$. Denoting by $\chi_1= u\mid_{[0,1]}$ and $\chi_2= v\mid_{[0,1]}$ we get
\begin{align*}
  \|(L(1)\chi_1- L(1)\chi_2)_{t}\|_{L^2(0,1;V^{'}_0)}^2 =& \int_{0}^{1}\|w_t(1+\theta)\|_{V_0^{'}}^2\,d\theta \\
   =& \int_{0}^{1}\left(\sup_{\substack{ \|\psi\|_{V_0} \leqslant 1}} |(w_t(1+\theta),\psi)_{H_0}|\right)^2\,d\theta \\
   \leqslant & \int_{0}^{1}(\alpha_1 \gamma)^2 \|\nabla w(1+\theta)\|_{H_0}^2\,d\theta,
\end{align*}
where $\alpha_1$ is the constant of continuous immersion $L^{p(x)}(\Omega)\hookrightarrow L^2(\Omega)$. Therefore
$$\|(L(1)\chi_1- L(1)\chi_2)_{t}\|_{L^2(0,1;V^{'}_0)} \leqslant \alpha_1\gamma \|\nabla w(1+ \cdot)\|_{L^2(0,1;H_0)}.
$$
In order to estimate the term $\|\nabla w(1+ \cdot)\|_{L^2(0,1;H_0)}$
we make the product $(\cdot, w)_{H_0}$
in the equation \eqref{lema3.eq1}, so,
\begin{equation}\label{lema3.eq2}
  (w_t,w)_{H_0}+ (A_0^{H_0}u - A_0^{H_0}v,w)_{H_0} =(Bu -Bv,w)_{H_0}.
\end{equation}

By Tartar's inequality, \eqref{desiTartar}, we can conclude 
\begin{align*}
  (A_0^{H_0}u -  &A_0^{H_0}v,w)_{H_0}=  \langle A_0u - A_0 v,w\rangle_{V_0^{'},V_0} \\
  = & \int_{\Omega_1} d_0(x)\left( |\nabla u|^{p(x)-2}\nabla u - |\nabla v|^{p(x)-2}\nabla v \right)\nabla w \, dx + \int_{\Omega_1}d_0(x)\eta(\nabla u- \nabla v)\nabla w \, dx \\
   & \hspace*{8cm} + \int_{\Omega} (|u|^{p(x)-2}u- |v|^{p(x)-2}v)w \,dx\\
  \geqslant & m_0 \int_{\Omega_1} \frac{2^{3-p(x)}}{p(x)} |\nabla w|^{p(x)} \,dx + m_0\eta\int_{\Omega_1} |\nabla w|^2 \,dx + \int_{\Omega} \frac{2^{3-p(x)}}{p(x)} |w|^{p(x)} \,dx\\
  \geqslant & m_0 \frac{2^{3-p^{+}}}{p^{+}} \int_{\Omega_1} |\nabla w|^{p(x)} \,dx + m_0\eta\int_{\Omega_1} |\nabla w|^2 \,dx + \frac{2^{3-p^{+}}}{p^{+}} \int_{\Omega} |w|^{p(x)} \,dx\\
  \geqslant &  m_0\eta\int_{\Omega_1} |\nabla w|^2 \,dx,
\end{align*}
once  $\rho_{p}(w)$ and $\rho_{p}(\nabla w)$
are non-negatives. So coming back to \eqref{lema3.eq2} 
\begin{align*}
   \frac{1}{2}\frac{d}{dt}\|w\|_{H_0}^2 +m_0\eta\int_{\Omega_1}|\nabla w|^2\,dx &\leqslant (w_t,w)_{H_0}+(A_0^{H_0}u - A_0^{H_0}v,w)_{H_0} =(Bu -Bv,w)_{H_0} \\
   & \leqslant  |(Bu -Bv,w)_{H_0}| \leqslant \|Bu -Bv\|_{H_0}\|w\|_{H_0} \\
   &\leqslant L_B \|u-v\|_{H_0}\|w\|_{H_0} = L_B \|w\|_{H_0}^2,
\end{align*}
that is,
\begin{equation}\label{lema3.eq3}
  \frac{1}{2}\frac{d}{dt}\|w(t)\|_{H_0}^2 +m_0\eta\int_{\Omega_1}|\nabla w(t)|^2\,dx \leqslant L_B \|w(t)\|_{H_0}^2.
\end{equation}

Neglecting the second term of the sum in \eqref{lema3.eq3} and integrating for $\theta$ varying in interval $[s,t]$, where $0\leqslant s < t$, we get
\begin{align*}
  \|w(t)\|_{H_0}^2 \leqslant \|w(s)\|_{H_0}^2 + 2L_B\int_{s}^{t}\|w(\theta)\|_{H_0}^2 \,d\theta,
\end{align*}
applying the Gronwall's inequality, it follows that
\begin{equation}\label{lema3.eq4}
  \|w(t)\|_{H_0}^2 \leqslant \|w(s)\|_{H_0}^2\exp(2L_B(t-s)), \quad \text{ for }\,\, 0\leqslant s <t.
\end{equation}
Returning the inequality \eqref{lema3.eq3} and integrating in interval $[\tau,2]$, $\tau \in [0,1]$, we get
\begin{align*}
  \|w(2)\|_{H_0}^2 +2m_0\eta\int_{\tau}^{2} \int_{\Omega_1}|\nabla w(\theta)|^2\,dx\,d\theta \leqslant 2L_B \int_{\tau}^{2}\|w(\theta)\|_{H_0}^2\,d\theta  +\|w(\tau)\|_{H_0}^2,
\end{align*}
that is,
\begin{equation}\label{lema3.eq5}
  2m_0\eta\int_{\tau}^{2} \int_{\Omega_1}|\nabla w(\theta)|^2\,dx\,d\theta \leqslant 2L_B \int_{\tau}^{2}\|w(\theta)\|_{H_0}^2\,d\theta  +\|w(\tau)\|_{H_0}^2.
\end{equation}
As $0\leqslant\tau\leqslant\theta\leqslant 2$ then by \eqref{lema3.eq4}, $$\|w(\theta)\|_{H_0}^2 \leqslant \|w(\tau)\|_{H_0}^2\exp(2L_B(\theta-\tau)),$$ and  integrating the last inequality for $\theta \in [\tau,2],$ we have
\begin{align*}
  \int_{\tau}^{2}\|w(\theta)\|_{H_0}^2 \,d\theta \leqslant  & \int_{\tau}^{2}\|w(\tau)\|_{H_0}^2\exp(2L_B(\theta-\tau))\,d\theta \\
  = & \frac{1}{2L_B}\|w(\tau)\|_{H_0}^2 (\exp(2L_B(2-\tau))-1)
\end{align*}
replacing in \eqref{lema3.eq5} and using the growth of the function $\exp(\cdot)$, we get

\begin{equation*}
  2m_0\eta\int_{\tau}^{2} \int_{\Omega_1}|\nabla w(\theta)|^2\,dx\,d\theta \leqslant \exp(4L_B)\|w(\tau)\|_{H_0}^2,
\end{equation*}
for all $\tau \in [0,1]$. For all $\tau\in [0,1]$, we conclude 
\begin{align*}
  \int_{0}^{1}\|\nabla w(1+\theta)\|_{L^2}^2\,d\theta =&\int_{0}^{1}\|\nabla w(1+\theta)\|_{H_0}^2\,d\theta = \int_{1}^{2}\|\nabla w(\theta)\|_{H_0}^2\,d\theta  \\
  \leqslant & \int_{\tau}^{2}\|\nabla w(\theta)\|_{H_0}^2\,d\theta \\
  \leqslant & \frac{\exp(4L_B)}{2m_0\eta}\|w(\tau)\|_{H_0}^2,
\end{align*}
and by integrating about $\tau$ in $[0,1]$, we obtain
\begin{align*}
  \|\nabla w(1 + \cdot)\|_{L^2(0,1;H_0)}^2=\int_{0}^{1}\|\nabla w(1+\theta)\|_{H_0}^2\,d\theta &\leqslant \int_{0}^{1} \frac{\exp(4L_B)}{2m_0\eta}\|w(\tau)\|_{H_0}^2\,d\tau \\&=\frac{\exp(4L_B)}{2m_0\eta} \|w\|_{L^2(0,1;H_0)}^2.
\end{align*}
So we have
\begin{align*}
  \|L(1)\chi_1- L(1)\chi_2\|_Y =&\|\nabla (L(1)\chi_1-L(1)\chi_2)\|_{L^2(0,1;H_0)} + \|(L(1)\chi_1-L(1)\chi_2)_t\|_{L^2(0,1; V_0^{'})} \\
  \leqslant & \|\nabla w(1 + \cdot)\|_{L^2(0,1;H_0)} + \alpha_1\gamma\|\nabla w(1 + \cdot)\|_{L^2(0,1;H_0)} \\
  \leqslant & (1+\alpha_1\gamma)\sqrt{\exp(4L_B)(2m_0\eta)^{-1}} \|w\|_{L^2(0,1;H_0)} \\
  =& \rho_1 \left(\int_{0}^{1}\|u(\theta)-v(\theta)\|_{H_0}^{2} \,d\theta\right)^{\frac{1}{2}}\\
  = & \rho_1 \left(\int_{0}^{1}\|\chi_1(\theta)-\chi_2(\theta)\|_{H_0}^{2} \,d\theta\right)^{\frac{1}{2}}\\
  = & \rho_1 \|\chi_1-\chi_2\|_{L^2(0,1;H_0)},
\end{align*}
where  $\rho_1= (1+\alpha_1\gamma)\sqrt{\exp(4L_B)(2m_0\eta)^{-1}}$.
\end{proof}
In a similar way, we have that the application
\begin{align*}
\textrm{e}: & \mathfrak{X} \rightarrow  H_0\\
       & \chi          \mapsto     \chi(1)
\end{align*}
is Lipschitz continuous. In fact, from \eqref{lema3.eq4} we get
\begin{align*}
\|w(1)\|_{H_0}^2 \leqslant \|w(\theta)\|_{H_0}^2\exp(2L_B(1-\theta)) \leqslant \|w(\theta)\|_{H_0}^2\exp(2L_B), \quad \forall \theta \in [0,1].
\end{align*}

Integrating this last inequality, to $\theta$ varying in $[0,1]$, we have
\begin{align*}
& \int_{0}^{1}\|w(1)\|_{H_0}^2 \,d\theta \leqslant \int_{0}^{1}\|w(\theta)\|_{H_0}^2\exp(2L_B)\,d\theta,
\end{align*}
that is,
\begin{align}\label{estimativa_w(1)}
 \|w(1)\|_{H_0} \leqslant \exp(L_B) \left( \int_{0}^{1}\|w(\theta)\|_{H_0}^2\,d\theta \right)^{\frac{1}{2}}= \exp(L_B) \|w\|_{L^2(0,1;H_0)}.
\end{align}
Therefore, given any $\chi_1,\chi_2 \in \mathcal{B}_0$ there are $u$ and $v$ solutions of \eqref{prob_P_0_abstrato} with $u(0),v(0)\in B_0$ such that $u|_{[0,1]}=\chi_1$ and $v|_{[0,1]}=\chi_2$, making $w=u-v$,  from \eqref{estimativa_w(1)} follows that
\begin{align*}
 \|\textrm{e}(\chi_1)- \textrm{e}(\chi_2)\|_{H_0} =& \|\chi_1(1)-\chi_2(1)\|_{H_0} = \|w(1)\|_{H_0}\\
 \leqslant & \exp(L_B) \|w\|_{L^2(0,1;H_0)}= \exp(L_B) \|u-v\|_{L^2(0,1;H_0)} \\
 =& \exp(L_B) \|\chi_1-\chi_2\|_{L^2(0,1;H_0)}.
\end{align*}

\begin{proposition} 
  Exists $c_3>0$ such that
  $$\|L(s)\chi_1-L(t)\chi_2\|_{L^2(0,1;H_0)} \leqslant c_3(|s-t|^{\frac{1}{2}}+ \|\chi_1-\chi_2\|_{L^2(0,1;H_0)} ),
  $$
  for all $t,s\in[0,1]$ and for all $\chi_1,\chi_2\in \mathcal{B}_0$.
\end{proposition}
\begin{proof}
 First, lets go to prove that exists $c(T)>0$ such that
  $$\|T_0(s)u_0-T_0(t)v_0\|_{H_0} \leqslant c(T)(|s-t|^{\frac{1}{2}}+ \|u_0-v_0\|_{H_0} ),
  $$
  for all $s,t\in[0,T]$ and $u_0,v_0\in B_0$. In fact, for any $s,t\in[0,T]$ and $u_0,v_0\in B_0$
  \begin{equation}\label{prop(A9)(A10).eq1}
    \|T_0(s)u_0-T_0(t)v_0\|_{H_0}\leqslant \|T_0(s)u_0-T_0(t)u_0\|_{H_0}+\|T_0(t)u_0-T_0(t)v_0\|_{H_0}
  \end{equation}
  It follows from the Fundamental Theorem of Calculus, Hölder's Inequality and Fubini's Theorem that
  \begin{align*}
    \|T_0(s)u_0 - T_0(t)u_0\|_{H_0}^2 & = \int_{\Omega}|T_0(s)u_0(x)-T_0(t)u_0(x)|^2 \,dx\\
    & = \int_{\Omega}\left| \int_{t}^{s}(T_0(\theta)u_0(x))_\theta \,d\theta \right|^2 \,dx \\
    &\leqslant \int_{\Omega}\left( \int_{t}^{s}|(T_0(\theta)u_0(x))_\theta| \,d\theta \right)^2 \,dx \\
    &\leqslant \int_{\Omega} \left(\int_{t}^{s}|(T_0(\theta)u_0(x))_\theta|^2 \,d\theta |s-t|\right) \,dx\\
   & \leqslant  |s-t|\int_{\Omega}\int_{t}^{s}|(T_0(\theta)u_0(x))_\theta|^2 \,d\theta \,dx \\
    &= |s-t|\int_{t}^{s}\int_{\Omega}|(T_0(\theta)u_0(x))_\theta|^2  \,dx \,d\theta\\
   & =  |s-t|\int_{t}^{s}\|(T_0(\theta)u_0)_\theta\|_{H_0}^2 \,d\theta\\
    & \leqslant  |s-t|\int_{0}^{T}\|(T_0(\theta)u_0)_\theta\|_{H_0}^2 \,d\theta\\
     & = |s-t|\|(T_0(\cdot)u_0)_t\|_{L^2(0,T;H_0)}^2.
  \end{align*}
  As $u_0\in B_0 \subset D(A_0^{H_0}),$ then by consequence of the Theorem 3.17 in \cite{BrezisF}
  \begin{equation*}
    \|(T_0(\cdot)u_0)_t\|_{L^2(0,T;H_0)}\leqslant c_1(T).
  \end{equation*}
  In addition, it follows from \eqref{lema3.eq4} that
  \begin{equation*}
    \|T_0(t)u_0-T_0(t)v_0\|_{H_0}\leqslant c_2(T)\|u_0-v_0\|_{H_0}.
  \end{equation*}
  Returning \eqref{prop(A9)(A10).eq1} follow that
  \begin{equation}\label{prop(A9)(A10).eq2}
    \|T_0(s)u_0-T_0(t)v_0\|_{H_0}\leqslant c_1(T) |s-t|^{\frac{1}{2}} +c_2(T)\|u_0-v_0\|_{H_0}.
  \end{equation}
 In its turn
  \begin{align*}
    \|L(s)\chi_1-L(t)\chi_2\|_{L^2(0,1;H_0)}^2=&\int_{0}^{1} \|T_0(s+\theta)u_0-T_0(t+\theta)v_0\|_{H_0}^2 \,d\theta  \\
    = & \int_{0}^{1} \|T_0(s)T_0(\theta)u_0-T_0(t)T_0(\theta)v_0\|_{H_0}^2 \,d\theta,
  \end{align*}
  note that $u_0=\chi_1(0)\in B_0$, $v_0=\chi_2(0)\in B_0$, and $B_0$ be positively invariant with respect to $\{T_0(t)\}_{t\geqslant 0}$, then $T_0(\theta) u_0,T_0(\theta) v_0\in B_0$. So from \eqref{prop(A9)(A10).eq2} we have
  \begin{align*}
    \|L(s)\chi_1-L(t)\chi_2\|_{L^2(0,1;H_0)}^2\leqslant & \int_{0}^{1} (c_1(1)+c_2(1))^2 ( |s-t|^{\frac{1}{2}} +\|T_0(\theta)u_0-T_0(\theta)v_0\|_{H_0})^2 \,d\theta \\
    \leqslant & \int_{0}^{1} 2^2(c_1(1)+c_2(1))^2 ( |s-t| +\|T_0(\theta)u_0-T_0(\theta)v_0\|_{H_0}^2) \,d\theta \\
    =&  4(c_1(1)+c_2(1))^2 ( |s-t| + \int_{0}^{1}\|T_0(\theta)u_0-T_0(\theta)v_0\|_{H_0}^2 \,d\theta )\\
    =& 4(c_1(1)+c_2(1))^2(|s-t|+ \|\chi_1-\chi_2\|_{L^2(0,1;H_0)}^2 ).
  \end{align*}
  Therefore, for $c_3=2(c_1(1)+c_2(1))$ follows the result.
\end{proof}

Now, we established the main result of this paper.
\begin{theorem}\label{Teorema-existencia-exponential-attractor}
 The dynamical systems associate to \eqref{prob_plambda}-\eqref{prob_pzero} possesses the global attractor $\mathcal{A}_\lambda$, for all $\lambda\in[0,1]$,  which is bounded in $L^2(\Omega)$. Moreover, for each $\lambda\in[0,1]$ there exists a positively invariant subset $B$ of $L^2(\Omega)$ such that $\mathcal{A}_{\lambda}\subset B$ and the dynamical systems admits an exponential attractor $\mathcal{E}_\lambda$.
\end{theorem}
\begin{proof}
In consequence of Theorem 2.5 in \cite{MalekPrazak} we conclude that the dynamics of systems $\{L(t)\}_{t\geqslant 0}$ admits an exponential attractor in $\mathcal{B}_0$. Also by Theorem 2.6 in \cite{MalekPrazak} we have $\{T_0(t)\}_{t\geqslant 0}$ has an exponential attractor in $\textrm{e}(\mathcal{B}_0)$. Proceeding in an analogous way we conclude the existence of an exponential attractor for the dynamics $\{T_\lambda(t)\}_{t\geqslant 0}$, where $\lambda \in (0,1]$. 
\end{proof}


\end{document}